\documentclass[12pt,a4paper]{amsart}
\usepackage[utf8]{inputenc}
\usepackage{tikz}
\usepackage[english]{babel}
\usepackage[T1]{fontenc}
\usepackage{amsmath}
\usepackage{amsfonts}
\usepackage{amssymb}
\usepackage{graphicx}
\usepackage[left=2.5cm,right=2.5cm,top=3cm,bottom=2.5cm]{geometry}
\usepackage{ulem}
\usepackage{color}
\newtheorem{defi}{Definition}[section]
\newtheorem{thm}{Theorem}[section]

\newtheorem{rem}{Remark}[section]

 \newtheorem{prop}{Proposition}[section]
\newtheorem{lem}{Lemma}[section]
\newtheorem{theorem}{Theorem}[section]

\newtheorem{conj}[theorem]{Conjecture}
\newtheorem{opn}[defi]{Open Problem}
\newtheorem{lemma}[theorem]{Lemma}

\newtheorem{cor}{Corollary}[section]

\usepackage{color}

\begin{document}

 \title [Fundamental spectral gap] {On the fundamental spectral gap of weighted Schr\"{o}dinger operators }

\address{ \newline
Team of Modeling and Scientific Computing, 
Department of Mathematics, 
Multidisciplinary Faculty of Nador, 
University of Mohammed First, Morocco}
\email{m.ahrami@ump.ac.ma}






\author{Mohammed Ahrami}
 \maketitle


 \maketitle

\begin{abstract}
We characterize the potential $V(x)$ that minimizes the fundamental spectral gap of weighted Schr\"{o}dinger operators on the interval $[0,\pi]$ subject to Dirichlet boundary conditions, under the constraint that the potential $V(x)$ is convex and that the weight function $w(x)$ is concave.

\end{abstract}



\section{Introduction}


In this article, we study the minimization problem of the {\it fundamental eigenvalue gap}
$\Gamma := \lambda_2 - \lambda_1$ of weighted Schr\"{o}dinger operators

\begin{equation}\label{ePnoP}
H(V,w) u :=  -u^{\prime \prime} +V(x) u  = \lambda  w(x) u,
\quad\quad x\in [0,\pi].
\end{equation}
on a finite interval $[0,\pi]$ with Dirichlet boundary conditions, i.e., $u(0) = u(\pi) = 0$.

First, we provide some motivation for studying the fundamental gap and the importance of its bounds. There are both physical and mathematical reasons to study the fundamental gap. For example, in quantum mechanics, the fundamental gap represents
the energy needed to achieve the first excited state from the ground state of the
particle described by the considered Schr\"odinger operators, the so-called excitation
energy. The fundamental gap is also of importance in quantum field theory and 
statistical mechanics. To give an example, Van den Berg in \cite{Berg} applied gap 
results of the Laplace operator to provide sufficient conditions for a free boson gas 
to fill the ground state alone under the thermodynamic limit macroscopically. 

 The optimal estimates of the fundamental spectral gap of Schr\"odinger operators have received considerable attention in the last decades, and have been extensively studied since the 1980s; cf. \cite{Ash1, Ash2,Ash3, Ash4, ahrami, Ahrami3, Horvath1,Horvath2,Lavine} with most prior works assuming that $w(x)=1$.

 In the unweighted case,
where \(w(x)=1\), classical results of Ashbaugh and Benguria \cite{Ash1} established
sharp lower bounds for the fundamental gap under symmetry and single-well
assumptions on the potential. Later, Lavine \cite{Lavine} proved that the constant potential minimizes the fundamental gap among
all convex potentials and satisfies the lower bound

\begin{equation}\label{conjecture1}
 \Gamma\geqslant \frac{3\pi^{2}}{l^{2}} 
\end{equation}

\noindent where $l$ denotes the length of the interval.

This result revealed a remarkable rigidity phenomenon
for Schr\"odinger operators with convex potentials and became one of the central results in one-dimensional
spectral gap theory.

Some time later, in \cite{Andrews}, Andrews and Clutterbuck showed that the
fundamental gap of a Dirichlet Schr{\"o}dinger operator with semi-convex potential $q$ on a bounded convex domain $\Omega\in \mathbb{R}^{d}$ with diameter $D$ satisfies the estimate \eqref{conjecture1}, thus proving the so-called fundamental gap conjecture in essentially full generality. 

For graphs, Ahrami et al. \cite{AHEK} proved an interesting phenomenon: on a metric tree graph the optimal
potential for Schr\"odinger operators with convex potentials must be piecewise affine and there are graphs where the constant potential
does not minimize.

More recently, Ahrami et al. \cite{AHRA} studied the problem (\ref{ePnoP}) for single-well potentials and single-barrier weights and proved that for weighted Schr\"{o}dinger operators, both the optimal potential and the optimal weight are step functions with a specific characterization.

Our goal is to determine
whether the classical phenomenon discovered by Lavine
persists in the weighted setting, namely whether constant potentials continue to minimize the fundamental gap in the presence of nontrivial weight. The presence of a nontrivial weight significantly changes the spectral structure of the problem and introduces new interactions between the geometry
of the potential and the inhomogeneity of the medium.

Our first result shows that any minimizing convex potential can be reduced to an affine potential without increasing the spectral gap. Similarly, every concave weight can be reduced to an affine weight. This reduction is obtained through a perturbative argument based on the Feynman-Hellmann formula together with qualitative properties of the first two eigenfunctions. We then analyze affine weights in greater detail and prove that, under a suitable quantitative condition on the weight,
\[p > 5m\pi,
\]
where \(w(x)=mx+p\), any gap-minimizing convex potential must necessarily be constant.


Compared with the previous work of Ahrami et al. \cite{AHRA} on single-well potentials and single-barrier functions, the present paper
reveals a fundamentally different optimization mechanism. In the
single-well setting, the minimizers are typically discontinuous step
functions. Furthermore, in this paper the
optimization procedure reduces the problem to affine structures and eventually
recovers constant potentials under suitable quantitative assumptions on the
weight. Thus, the present work establishes a new bridge between the direct
optimization developed in \cite{ElAHar,
AHRA} and Lavine's result \cite{Lavine}.

The paper is organized as follows. In Section 2, we recall several preliminary properties of the fundamental eigenvalue gap, including a Feynman-Hellmann formula for eigenvalue variations and results on the monotonicity of the eigenfunction ratio $\frac{u_2}{u_1}$. In Section 3 we then establish in Corollary \ref{existence} general existence results by invoking compactness of the sets of functions with the properties of interest,
which in turn follow from a general argument using either the Blaschke selection principle. With these results in hand, we turn to characterizing the optimizers in Theorem \ref{step} and finally in Theorem \ref{step2}, we prove that for sufficiently large affine weights, the minimizing convex potential is necessarily constant.

\section{Preliminaries and basics}
Throughout the paper, we study the weighted Schr\"{o}dinger eigenvalue problem
\begin{equation*}\label{eq}
H(V,w) u :=  -u^{\prime \prime} +V(x) u  = \lambda  w(x) u
\end{equation*}

subject to Dirichlet boundary conditions $u(0)=u(\pi)=0$
with a measurable bounded potential $V:[0,\pi] \longrightarrow
 \mathbb{R}$ and measurable bounded weight $w:[0,\pi] \longrightarrow
 \mathbb{R^+}$.

\noindent Under the above assumptions on $V$ and $w$, the operator $H(V,w)$ is self-adjoint with compact resolvent and therefore possesses a purely discrete spectrum; see, e.g., \cite{AHRA,Wei}.


\noindent We denote the eigenvalues of $H$ by

$$
0<\lambda_1(V, w)<\lambda_2(V, w)<\lambda_3(V, w)<...
$$
indicating the dependence on the potential $V$ and the weight $w$. Furthermore, all eigenvalues are simple; see, e.g., \cite{Zett, Wei}.

The main object of interest in this paper is the fundamental gap $\Gamma[V, w]$ defined by  

 $$\Gamma[V, w] = \lambda_2(V, w)- \lambda_1(V, w).$$
It follows immediately from the definition that the fundamental gap of a weighted Schr\"{o}dinger operators subject to
Dirichlet boundary conditions on an interval $[0,\pi]$ is invariant under the addition of a constant, i.e.,
$ \Gamma[V+c,w]= \Gamma[V,w]$.
As a consequence of this remark, which will be very important later on, if we have 
to deal with affine potentials, we only need to consider potentials of the form
$V(x)=ax $  with $a>0$. 

We recall several standard facts from spectral gap theory; see for instance \cite{Ash1,Lavine,ElAHar}.



\begin{lem}\label{F-H}
Suppose that $V(.,\kappa)$ and $w(.,\kappa)$ are one-parameter families of real-valued, locally $L^1$ functions, \( C^1 \) in \( \kappa \) for \( \kappa \) in some real interval, with
$\inf_{x} V(x,\kappa) > - \infty$, $C \ge w(x,\kappa) \ge \frac 1 C$ for some $C > 0$, and
$\frac{\partial V}{\partial \kappa}(x,\kappa)$
and $\frac{\partial w}{\partial \kappa}(x,\kappa) \in L^1(0, \pi)$. Without loss of generality we standardize the first two normalized eigenfunctions $u_{1,2}(x)$ associated with 
$\lambda_{1,2}$
so that
$u_{1,2}(x) > 0$ for $0 < x < \epsilon$ for some $\epsilon$.  Then
\begin{enumerate}
\item \[
\frac{d\lambda_{n}(\kappa)}{d\kappa}=-\lambda_{n}\int_{0}^{\pi}\frac{\partial w}{\partial \kappa}(x,\kappa)u_{n}^{2}(x,\kappa)dx+\int_{0}^{\pi}\frac{\partial V}{\partial \kappa}(x,\kappa)u_{n}^{2}(x,\kappa)dx .
\]
\item
$\dfrac{u_{2}}{u_{1}} $ is decreasing on $(0,\pi)$.

\item The equation $ \vert u_{1}(x)\vert=\vert u_{2}(x)\vert $ has
either one or two solutions on $(0,\pi)$.
 \item There exist two points $x_{-}$ and $x_{+}$,
  $ 0 \leq x_{-}< x_{+} \leq \pi $, at least one of which is interior to $(0, \pi)$, such that 
$u_{1}^{2}(x)>u_{2}^{2}(x)$ on $(x_{-},x_{+})$ and $u_{1}^{2}(x)\le u_{2}^{2}(x)$ on $(x_{-},x_{+})^c$.
\item
The equation $ \lambda_1 \vert u_{1}^2(x)\vert=\lambda_2 \vert u_{2}^2(x)\vert $ has
either one or two
solutions on $(0,\pi)$.
 \item 
There exist two points $\widehat{x}_{-}$ and $\widehat{x}_{+}$,
  $ 0 \leq \widehat{x}_{-}< \widehat{x}_{+} \leq \pi $, at least one of which is interior to $(0, \pi)$, such that 
$\lambda_1 u_{1}^{2}(x)> \lambda_2 u_{2}^{2}(x)$ on $(\widehat{x}_{-},\widehat{x}_{+})$ and $\lambda_1 u_{1}^{2}(x)\le \lambda_2 u_{2}^{2}(x)$ on $(\widehat{x}_{-},\widehat{x}_{+})^c$.
\end{enumerate}

\end{lem}

\noindent
{For the proof,  see \cite{AHRA}.}

\begin{lem}\label{G'''}
Let $G:[0,\pi]\to\mathbb{R}$ be a function such that $G^{''}(x)$ is absolutely continuous. If $u$ satisfies \eqref{ePnoP}, where $V(x),w(x)$ are differentiable, then
\begin{align*}
G(\pi)[u'^{2}(\pi)&+(\lambda w(\pi)-V(\pi))u^{2}(\pi)]-G(0)[u'^{2}(0)+(\lambda w(0)-V(0))u^{2}(0)]\\
&\quad\quad\quad\quad+ \dfrac{1}{2}[G''(\pi)u^{2}(\pi)-G''(0)u^{2}(0)]\\
&=\int_{0}^{\pi}\left[2G'(\lambda w-V)+G(\lambda w'-V')+\frac{G'''}{2}\right]
u^{2}(x)dx.
\end{align*}
\end{lem}

\begin{proof}  
We have
\begin{align*}
G(\pi)[u'^{2}(\pi)&+(\lambda w(\pi)-V(\pi))u^{2}(\pi)]-G(0)[u'^{2}(0)+(\lambda w(0)-V(0))u^{2}(0)]\\
 &=\int_{0}^{\pi}\frac{dG(x)}{dx}[u'^{2}(x)+(\lambda w(x)-V(x))u^{2}(x)]dx  \\
&=\int_{0}^{\pi}G'[u'^{2}+(\lambda w-V)u^{2}]+G[2u'u''+(\lambda w'-V')u^{2}+2(\lambda w-V)uu']dx.\\
&=\int_{0}^{\pi}G'[u'^{2}+(\lambda w-V)u^{2}]+G[-2(\lambda w-V)uu'+(\lambda w'-V')u^{2}+2(\lambda w-V)uu']dx\\
&=\int_{0}^{\pi}G'[u'^{2}+(\lambda w-V)u^{2}]+G(\lambda w'-V')u^{2}dx
\end{align*}
and 
\begin{align*}
0 &=\int_{0}^{\pi}(G'uu')'dx\\
&=\int_{0}^{\pi}(G''uu'+G'u'^{2}+G'uu'')dx\\
&=\int_{0}^{\pi}[G''uu'+G'u'^{2}-G'(\lambda w-V)u^{2}]dx\\
&=\frac{1}{2}[G''(\pi)u^{2}(\pi)-G''(0)u^{2}(0)]-\frac{1}{2}\int_{0}^{\pi}G'''u^{2}dx-\int_{0}^{\pi}G'(\lambda w-V)u^{2}dx+\int_{0}^{\pi}G'u'^{2}dx.
\end{align*}
(Since $G''$ is assumed absolutely continuous $G'''$ is defined a.e. and integration by parts is valid.)  Thus
$$\frac{1}{2}[G''(\pi)u^{2}(\pi)-G''(0)u^{2}(0)]=\int_{0}^{\pi}
\left[\frac{G'''u^{2}}{2}+G'(\lambda w-V)u^{2}-G'u'^{2}\right]dx.
$$
\noindent Combining the previous identities yields the desired formula
 $$ G(\pi)[u'^{2}(\pi)+(\lambda w(\pi)-V(\pi))u^{2}(\pi)]-G(0)[u'^{2}(0)+(\lambda w(0)-V(0))u^{2}(0)]$$
$$ + \dfrac{1}{2}[G''(\pi)u^{2}(\pi)-G''(0)u^{2}(0)]$$
$$=\int_{0}^{\pi}\left[\frac{G'''u^{2}}{2}+G'(\lambda w-V)u^{2}-G'u'^{2}+G(\lambda w'-V')u^{2}+G'((\lambda w-V)u^{2}+u'^{2})\right]dx. $$
\end{proof}




\section{Characterization of optimizers}
 In this section, we look at the explicit form of the minimizing potential of the fundamental spectral gap for weighted Schr\"{o}dinger operators under Dirichlet boundary conditions {\eqref{ePnoP}} 
{closely following the strategy in \cite{ElAHar}.}

\subsection{The class of convex potentials and concave weight functions}

 We define here several classes of functions that will play a role below.

\begin{defi}

For $1<M \le \infty$ and $0 < N_{<,>} < \infty $, let
$$ \mathcal{C}_M=\lbrace V(x): 0\le V(x)\le M : V(x) 
\textrm{ is convex on }[0,\pi] \rbrace .$$
and
$$ \mathcal{S}_{N_{<,>}}=\lbrace w(x):  N_< \leq w(x) \leq N_>: w(x) 
\textrm{ is concave on }[0,\pi] \rbrace .$$
\end{defi}

\begin{defi}
Consider the weighted Schr\"{o}dinger operators with Dirichlet boundary conditions
\begin{align}\label{s-l}
    {-u''+V(x)u=\lambda w(x)u}&\\
    u(0)=u(\pi)=0&,\nonumber
\end{align}                                                   
where $V$ is a bounded measurable convex function and $w$ is a bounded measurable concave density.
If there exist $V_{*} \in \mathcal{C}_M$ and 
$w_{*} \in  \mathcal{S}_{N_{<,>}}$ such that
\[
\Gamma(V_{*},w_{*})= \inf \lbrace
  \Gamma(V,w),V \in \mathcal{C}_M,
~w\in \mathcal{S}_{N_{<,>}} \rbrace,
\]
then we call the function $V_{*}$ an {\rm optimal potential} and the function  $w_{*}$ an {\rm optimal density} for problem \eqref{s-l}.
\end{defi}

\begin{conj}
We consider the weighted Schr\"{o}dinger operators

\begin{equation}\label{eq12}
\left\{
\begin{array}{lr}
H(V,w) u :=  -u^{\prime \prime} +V(x) u  = \lambda  w(x) u
 \\
u(0)=u(\pi)=0 & 
\end{array}
\right.
\end{equation}
For a given bounded measurable, strictly positive weight function $w$, any 
gap-minimizing
optimal potential $V_{*}\in \mathcal{C}_M $ is constant.
\end{conj}

\begin{rem}\label{rem1}
The assumptions of boundedness and measurability ensure that the weighted Schr\"{o}dinger operators $H(V,w)$  is self-adjoint with a discrete spectrum, as discussed in \cite{Wei}.
\end{rem}

{As in \cite{ElAHar}, we will use compactness of the sets $\mathcal{C}_M$ and $\mathcal{S}_{N_{<,>}}$ , following from a Blaschke Selection principle: }

\begin{prop}\label{compact}
For any sequence $f_{n}\in\Lambda$, where $\Lambda=\mathcal{C}_M$ or
respectively $\mathcal{S}_{N_{<,>}}$ with any fixed positive $M, N_{<,>}$, there exist a subsequence $f_{n_{k}}$ and a function $f_{*}$ such that  $f_{n_{k}}(x)\longrightarrow f_{*}(x) \in\Lambda $ for a.e. $x$ with $f_{*} \in \mathcal{C}_M$ or respectively $\mathcal{S}_{N_{<,>}}$.
\end{prop}

\noindent
{For the proof,  see \cite{ElAHar},
Proposition $2.1$.}

\begin{cor}\label{existence}
There exist a potential $V_{*}\in \mathcal{C}_M $ and a density $w_{*}\in \mathcal{S}_{N_{<,>}}$ that minimize $\Gamma[V,w]$.
\end{cor}

\begin{proof}  
The argument is identical to the proof of Corollary 3.1 in \cite{AHRA}, except that, as mentioned, one replaces the Helly compactness theorem with the Blaschke selection principle.
\end{proof}  

We can now seek the explicit form of optimal potential  $V_{*}\in \mathcal{C}_M $ of the problem  
{\eqref{s-l}}. We follow the strategy of \cite{AHRA}.

\begin{thm}\label{step}
For a given bounded measurable, strictly positive weight function $w$, and for every convex 
and non-affine potential $V$ (i.e. not of the form $ax+b$), there exists a linear potential $V_{a}=ax$ such that
$$
 \Gamma[V,w] \geq  \Gamma[V_{a},w].$$
Alternatively, for a given  
bounded measurable potential function $V$, and for every concave and not-affine
weight $w$ (i.e. not of the form $mx+p$), there exists an affine weight $w_m(x)=mx+p$ such that 
$$
 \Gamma[V,w] \geq  \Gamma[V,w_m].$$  
\end{thm}

\begin{proof}  

We assume that $w$ is fixed. 
Let $V$ be a convex non-affine potential and $L_{V}(x)=ax+b$ the affine potential 
such that $V(x_{\pm})=L_{V}(x_{\pm})$.

\noindent By Lemma \ref{F-H} there exist $x_{\pm}$: $0\leqslant x_{-}<x_{+}\leqslant \pi$, for which 
\begin{align*}
    &u_{2}^{2}(x)>u_{1}^{2}(x) \textrm{ on } (0,x_{-})\cup (x_{+},\pi)\quad\textrm{(one of these intervals may be vacuous)}\\
    &u_{1}^{2}(x)>u_{2}^{2}(x) \textrm{ on } (x_{-},x_{+}).
\end{align*}
 
\noindent By the convexity of $V$,
\begin{gather*}
 V-L_{V}\geq 0 \quad \text{on } (0,x_{-})\cup (x_{+},\pi), \\
 V-L_{V}\leq 0 \quad \text{on } (x_{-},x_{+}). 
\end{gather*}
Hence 
$$
\int_0^{\pi}(V(x)-L_{V}(x))(u_2^2(x)-u_1^2(x))dx >0.
$$
Let
\[
 L_{V}(\theta)=\theta V+(1-\theta)L_{V}\quad\text{for }\theta\in(0,1).
\]
 Therefore $\dot{L_{V}}(\theta)=V-L_{V}$.
Then by the Feynman-Hellman formula,
$$
\frac{d\Gamma(L_{V}(\theta))}{d\theta}=
\int_0^{\pi}(V(x)-L_{V}(x))(u_2^2(x)-u_1^2(x))dx >0.
$$
Integrating this inequality with respect to $\theta$ over $(0,1)$, we find that
$$
\int_0^{1}\frac{d\Gamma(L_{V}(\theta))}{d\theta}=\Gamma[V,w]-\Gamma[L_{V},w]\geq 0.
$$
Then $$ \Gamma[V,w]\geq \Gamma[L_{V},w]=\Gamma[ax,w].$$

We now prove the second part of the statement by a similar argument applied to $w$ for fixed $V$.
Let $w$ be a concave, not-affine weight, and $w_{m}=mx+p$ be the affine weight such that $w(\widehat{x}_{\pm})=w_{m}(\widehat{x}_{\pm})$.

Applying Lemma \ref{F-H} once again, we find that 
there exist $\widehat{x}_{\pm}$: $0\leqslant \widehat{x}_{-}<\widehat{x}_{+}\leqslant \pi$, satisfying 
\begin{align*}
    &\lambda_{2} u_{2}^{2}(x)>\lambda_{1} u_{1}^{2}(x) \textrm{ on } (0,\widehat{x}_{-})\cup (\widehat{x}_{+},\pi)\\
    &\lambda_{1}u_{1}^{2}(x)>\lambda_{2}u_{2}^{2}(x) \textrm{ on } (\widehat{x}_{-},\widehat{x}_{+}).
\end{align*}
\noindent By the concavity of $w$  
  $$\displaystyle  \left\{
    \begin{array}{ll}
   
   w-w_{m} \leqslant  0 \qquad\text{on}~~(0,\widehat{x}_{-})\cup (\widehat{x}_{+},\pi),
   \\
    w-w_{m}\geqslant 0 \qquad\text{on}~~(\widehat{x}_{-},\widehat{x}_{+}).                                                                                      \end{array}                                                                                \right.$$
Therefore
$$
\int_{0}^{\pi}(w-w_{m})(\lambda_2 u_{2}^{2}-\lambda_1 u_{1}^{2})dx <0.
$$ Let 
 $$ L_w(\theta)=\theta w+(1-\theta)w_{m} , ~~~ \theta\in(0,1).$$\\
Consequently
 $$\dot{L_{w}}(\theta)=w-w_{m}.  $$
 Then by the Feynman-Hellman formula,
$$\frac{d\Gamma(L_w(\theta))}{d\theta}=(\lambda_1-\lambda_2)
\int_{0}^{\pi}(w-w_{m})(\lambda_2 u_{2}^{2}-\lambda_1 u_{1}^{2})dx >0.
$$

\noindent Integrating this inequality with respect to $\theta$ over $(0,1)$, we find

$$
\int_{0}^{1}\dfrac{d\Gamma(L_w(\theta))}{d\theta}=\Gamma[V,w]-\Gamma[V,w_{m}]\geqslant 0.
$$
Consequently
$$
\Gamma[V,w]\geqslant \Gamma[V,w_m].
$$

\end{proof}

 In the final part of this section, the aim is to analyze the special case of the optimal potential of the form $V(x)=ax$ with $a\geq0.$ 

\noindent  In the context of single-well potentials and single-barrier weights, it was shown in \cite{AHRA} that optimizers have the following
characterization:
\begin{enumerate}
\item
$V_{*}(x) = 0$ a.e. on a
connected component of $\{x: u_2^2(x) > u_1^2(x)\}$, where
$u_{1,2}$ are the eigenfunctions associated with 
$\lambda_{1,2}$, while on the complement of that interval, 
$V_{*}(x) = M$ 
a.e.  
\item
$w_{*}(x) = N_>$ a.e. on a
connected component of $\{x: \lambda_2 u_2^2(x) > \lambda_1 u_1^2(x)\}$ and on the complement of that interval there exists a constant $\mathcal{M}$ such that $w_{*}(x) = \mathcal{M}$ where $N_< \leq \mathcal{M}<N_>$. 
\end{enumerate}
If $w(x)=1$, Lavine’s \cite{Lavine} result can be reformulated equivalently for weighted Schr\"{o}dinger operators as follows: $$
\Gamma[V,1]\geqslant \Gamma[0,1].
$$

\begin{lemma}\label{lemm}
    Let the first two normalized eigenfunctions $u_{1,2}(x)$ associated with the first eigenvalues $\lambda_{1,2}$ of problem (\ref{s-l}). Then
    \[
\int_0^{\pi}x^2(u_2^2(x) -u^2_1(x))dx\leq\frac{2\pi^2}{\min_{x} w(x)}.
\]

\end{lemma}

\begin{proof}
\noindent Let $w \in  \mathcal{S}_{N_{<,>}}$, since $w$ is concave on the compact interval $[0,\pi]$, it follows that $w$ is continuous and hence attains its minimum. In addition, the strict positivity of $w$ yields \[ \min_{x\in [0,\pi]} w(x)>0. \]
\noindent  Let the first two normalized eigenfunctions $u_{1,2}(x)$ associated with the first eigenvalues $\lambda_{1,2}$ of problem (\ref{s-l}), 
we have $\int_0^\pi w(x)u_n^2(x)dx=1$. Hence 
\[
\int_0^{\pi}u_n^2(x)dx\leq
\frac{1}{\min_{x} w(x)}\int_0^{\pi}u_n^2(x)w(x
)dx=\frac{1}{\min_{x} w(x)}.
\]
Therefore 

\[
\int_0^{\pi}\mid u_2^2(x)-u_1^2(x)\mid dx\leq
\int_0^{\pi}( u_2^2(x)+u_1^2(x)) dx\leq
\frac{2}{\min_{x} w(x)}.
\]
Since 
\[
\left| \int_0^{\pi}x^2 (u_2^2(x)-u_1^2(x))dx \right| \leq
\frac{2\pi^2}{\min_{x} w(x)}.
\]
Then
    \[
\int_0^{\pi}x^2(u_2^2(x) -u^2_1(x))dx\leq\frac{2\pi^2}{\min_{x} w(x)}.
\]
    
\end{proof}
\begin{rem}\label{rem12}
    In particular, for the affine weight $w(x)=mx+p$, $m>0$, in the interval $[0,\pi]$, the function $w$ increases since $w^{\prime}(x)=m>0$.
Therefore $\min_{x} w(x)=w(0)=p.$ Then

\[
\int_0^{\pi}x^2(u_2^2(x) -u^2_1(x))dx\leq\frac{2\pi^2}{p}.
\]
\end{rem}

\begin{thm}\label{step2}
Let $w(x)=mx+p$ be a fixed positive affine weight function with $m>0$. Assume that
\[
p>5m\pi
\]

\noindent Then any minimizer of the fundamental gap
\[
\Gamma[V,w] = \lambda_2(V,w) - \lambda_1(V,w)
\]
over the class of convex potentials $ \mathcal{C}_M$ must be constant.

    
\end{thm}
\begin{proof}  
By Theorem \ref{step}, for any convex potential $V$, there exists an affine potential
\[
V_a(x) = ax 
\]
such that
\[
\Gamma[V,w] \geq \Gamma[V_a,w].
\]


\noindent Let the first two normalized eigenfunctions $u_{1,2}(x)$ associated with the first eigenvalues $\lambda_{1,2}$ of problem (\ref{s-l}). 

\noindent We consider a family of potentials $V^{\theta}=\theta ax$. 
 From Lemma \ref{F-H}, we have

\begin{align*}
&\frac{d(\lambda_2(V^{\theta})-\lambda_1(V^{\theta}))}{d\theta}\\
&=\int_0^{\pi} \frac{\partial V^{\theta}}{\partial \theta}[u_2^2(x)-u_1^2(x)]dx\\
 &=a\int_0^{\pi}x[u_2^2(x)-u_1^2(x)]dx.
\end{align*}

\noindent Suppose that a critical point of the gap occurs
at some $a>0$, then
$$
\int_0^{\pi}x(u_2^2(x)-u_1^2(x))dx =0.
$$
\noindent Applying Lemma \ref{G'''} for $G(x)=x$, we have
\begin{align*}
&\pi[(u_n^{\prime}(\pi))^{2}+(\lambda_n w(\pi)-a\pi )u_n^{2}(\pi)]\\
 &=\int_0^\pi [2(\lambda_n w-ax)+x(\lambda_n w^{\prime}-a)]u_n^2
(x)dx.
\end{align*}

\noindent Since $u_n$ satisfies the Dirichlet boundary conditions $u_n(0)=u_n(\pi)=0$, 

$$\pi(u_n^{\prime}(\pi))^{2}
 =\int_0^\pi [2(\lambda_n w-ax)+x(\lambda_n w^{\prime}-a)]u_n^2
(x)dx.
$$ 

\noindent Similarly for $G(x)=x^2$ and from the Dirichlet boundary conditions, we find that
$$\pi^2(u_n^{\prime}(\pi))^{2}\\
=\int_0^\pi [4x(\lambda_n w-ax)+x^2(\lambda_n w^{\prime}-a)]u_n^2(x)dx.$$
\noindent Consequently
\begin{align*}
&\int_0^\pi [2(\lambda_1 w-ax)+x(\lambda_1 w^{\prime}-a)]u_1^2
(x)dx\\
&=\frac{1}{\pi}\int_0^\pi [4x(\lambda_1 w-ax)+x^2(\lambda_1 w^{\prime}-a)]u_1^2(x)dx
\end{align*}
\noindent and
\begin{align*}
&\int_0^\pi [2(\lambda_2 w-ax)+x(\lambda_2 w^{\prime}-a)]u_2^2
(x)dx\\
&=\frac{1}{\pi}\int_0^\pi [4x(\lambda_2 w-ax)+x^2(\lambda_2 w^{\prime}-a)]u_2^2(x)dx
\end{align*}

\noindent  Subtracting the two identities corresponding to $n=1$ and $n=2$, and using $
\int_0^{\pi}x(u_2^2(x)-u_1^2(x))dx =0$, we obtain

\begin{align*}
&  
\frac{2\pi}{5}(\lambda_2-\lambda_1)\\
&=\left(m(\lambda_2-\lambda_1) -a\right)\int_0^{\pi}x^2(u_2^2(x) -u^2_1(x))dx.\\
\end{align*}

\noindent Hence
\begin{align*}
&  
(\lambda_2-\lambda_1)  \left(\frac{2\pi}{5}-m\int_0^{\pi}x^2(u_2^2(x) -u^2_1(x))dx\right)\\
&=-a\int_0^{\pi}x^2(u_2^2(x) -u^2_1(x))dx.\\
\end{align*}

\noindent Let
\[
\phi(x) = x^2 - Ax - B
\]
\noindent chosen such that
\[
\phi(x_-) = \phi(x_+) = 0.
\]

\noindent Then

$$\displaystyle  \left\{
    \begin{array}{ll}
   
   \phi(x) \leqslant  0 \qquad\text{on}~~(x_-,x_+)^c,
   \\
   \phi(x)\geqslant 0 \qquad\text{on}~~(x_-,x_+).                                                                                      \end{array}                                                                                \right.$$

\noindent We obtain
\[
\int_0^\pi x^2 (u_2^2 - u_1^2)\,dx
= \int_0^\pi \phi(x)(u_2^2 - u_1^2)\,dx
+ A \int_0^\pi x(u_2^2 - u_1^2)\,dx
+ B \int_0^\pi (u_2^2 - u_1^2)\,dx>0.
\]

\noindent By Remark \ref{rem12} \[
\int_0^{\pi}x^2(u_2^2(x) -u^2_1(x))dx\leq\frac{2\pi^2}{p}.
\]
It follows that \[\frac{2\pi}{5m}
\leq 
 \frac{2\pi^2}{p}.
\]
But the assumption
\[
p>5m\pi
\]
which yields a contradiction, since the left-hand side is positive. Then the gap-minimizing
optimal potential $V_{*} \in \mathcal{C}_M $ is constant.
\end{proof}  




\subsection{Discussion and Open Problems}

In the present paper, we extended several classical results concerning the minimization of the fundamental spectral gap for Schr\"odinger operators to the weighted setting. More precisely, we proved that every minimizing convex potential can be reduced to an affine potential, while every minimizing concave weight can be reduced to an affine weight. Furthermore, under the quantitative condition
\[
p>5m\pi,
\]
for affine weights of the form
\[
w(x)=mx+p,
\]
we established that any minimizing convex potential must necessarily be constant.

These results indicate that the rigidity phenomenon discovered by Lavine in the unweighted case persists, at least partially, in the weighted framework. Nevertheless, several important questions remain open.

First, it is not clear whether the condition
\[
p>5m\pi
\]
is optimal or merely technical. The proof relies on an integral estimate involving the first two eigenfunctions, and it is conceivable that sharper estimates may improve this condition or possibly remove it altogether. This naturally leads to the following question:

\begin{opn}    
Let $w$ be a positive concave weight on $[0,\pi]$. Does the constant potential minimzer the fundamental spectral gap over the class of convex potentials. Equivalently, 
$$ \Gamma[V,w]\geq \Gamma[0,w]
$$
for every convex potential $V$.
\end{opn}

\subsection*{Acknowledgements} The author would like to thank Joachim Kerner (FernUniversität in Hagen) and James B. Kennedy (Universidade de Aveiro) for reading the manuscript and useful remarks.

\subsection*{Funding} This research was not supported by any sponsor or funder.






\end{document}